\documentclass[12pt]{article}
\usepackage[centertags]{amsmath}
\usepackage{amsfonts}
\usepackage{amssymb}
\usepackage{amsthm}

\theoremstyle{plain}
\newtheorem{thm}{Theorem}[section]

\newtheorem{lem}[thm]{Lemma}

\theoremstyle{definition}
\newtheorem{defn}[thm]{Definition}

\newtheorem{rem}[thm]{Remark}
\theoremstyle{remark}
\numberwithin{equation}{section}

\newcommand{\beast}{\begin{eqnarray*}}
\newcommand{\eeast}{\end{eqnarray*}}

\title{$3$-bounded Property in a Triangle-free Distance-regular Graph\thanks{Research partially supported
by the NSC grant 95-2115-M-009-002 of Taiwan R.O.C..}}

\author{Yeh-jong Pan\footnote{Department of Applied Mathematics National Chiao Tung University 1001 Ta Hsueh
Road Hsinchu, Taiwan 300, R.O.C..} \and Chih-wen
Weng\footnote{Department of Applied Mathematics, National Chiao
Tung University, Taiwan R.O.C..}}

\begin{document}
\maketitle

\bibliographystyle{plain}



\bigskip

\begin{abstract}
Let $\Gamma$ denote a distance-regular graph with classical
parameters $(D, b, \alpha, \beta)$ and $D\geq 3$. Assume the
intersection numbers $a_1=0$ and $a_2\not=0$. We show $\Gamma$ is
$3$-bounded in the sense of the article [$D$-bounded
distance-regular graphs, European Journal of Combinatorics(1997)18,
211-229].
\bigskip

\end{abstract}


\section {Introduction}\label{s1}

Let $\Gamma=(X,R)$ be a distance-regular graph with diameter
$D\geq3$ and distance function $\partial$. Recall that a sequence
$x$, $y$, $z$ of vertices of $\Gamma$ is {\it geodetic} \rm whenever
\begin{equation*}
  \partial (x,y) + \partial (y,z)=\partial (x,z).
\end{equation*}
A sequence $x$, $y$, $z$ of vertices of $\Gamma$ is {\it
weak-geodetic} whenever
\begin{equation*}
  \partial (x,y)+\partial (y,z)\leq \partial (x,z)+1.
\end{equation*}

\begin{defn}
A subset $\Omega \subseteq X$ is {\it weak-geodetically closed} if
for any weak-geodetic sequence $x$, $y$, $z$ of $\Gamma$,
$$x,\,z\in \Omega \Longrightarrow y\in \Omega.$$
\end{defn}

Weak-geodetically closed subgraphs are called {\it strongly closed
subgraphs} in \cite{s:95}. We refer the reader to \cite{sy:80,
bw:83, Iva:89, s:96, w:98, h:99} for the information of
weak-geodetically closed subgraphs.

\begin{defn}\label{def1.2}\noindent $\Gamma$ is said to be {\it i-bounded} whenever for
all $x,y\in X$ with $\partial(x,y)\leq i$, there is a regular
weak-geodetically closed subgraph of diameter $\partial(x,y)$ which
contains $x,y$.
\end{defn}
\medskip

The properties of $D$-bounded distance-regular graphs were studied
in \cite{w:97}, and these properties were used in the
classification of classical distance-regular graphs of negative
type \cite{w:99}. Before mention of our main result we mention one
more definition.
\medskip

By a {\it parallelogram of length $i$}, we mean a $4$-tuple $xyzw$
consisting of vertices of $\Gamma$ such that
$\partial(x,y)=\partial(z,w)=1$, $\partial(x,z)=i$, and
$\partial(x,w)=\partial(y,w)=\partial(y,z)=i-1$.
\medskip

It was proved that if $a_1=0$, $a_2\neq0$ and $\Gamma$ contains no
parallelograms of length $3$, then $\Gamma$ is $2$-bounded
\cite[Proposition~6.7]{w:98}, \cite[Theorem~1.1]{s:96}.  The
following theorem is our main result.
\medskip

\begin{thm}\label{main} Let $\Gamma$ denote a distance-regular graph
with classical parameters $(D, b, \alpha, \beta)$ and $D\geq 3$.
Assume the intersection numbers $a_1=0$ and $a_2\not=0$. Then
$\Gamma$ is $3$-bounded.
\end{thm}

Note that if $\Gamma$ has classical parameters $(D, b, \alpha,
\beta)$ with  $D\geq 3$, $a_1=0$ and $a_2\not=0,$  then $\Gamma$
contains no parallelograms of any length. See
\cite[Theorem~1.1]{plw:06} or Theorem~\ref{N3.3} in this article.

\section{Preliminaries}\label{s2}

  In this section we review some definitions, basic concepts and some previous results concerning distance-regular
graphs. See Bannai and Ito \cite{BanIto} or Terwilliger \cite{t:95}
for more background information.
\medskip

  Let $\Gamma$=($X$, $R$) denote a finite undirected, connected graph without
loops or multiple edges with vertex set $X$, edge set $R$, distance
function $\partial$, and diameter $D$:={\rm max}$\{\,\partial
(x,y)\mid x,y\in X\}$. By a {\it pentagon}, we mean a $5$-tuple
$x_{1}x_{2}x_{3}x_{4}x_{5}$ consisting of vertices in $\Gamma$ such
that $\partial(x_{i},x_{i+1})=1$ for $1\leq i\leq 4$ and
$\partial(x_{5},x_{1})=1$.
\bigskip

  For a vertex $x\in X$ and an integer $0\leq i\leq D$, set $\Gamma_{i}(x):=\{\,z\in X\mid\,\partial(x, z)=i\}$.
The {\it valency} $k(x)$ of a vertices $x\in X$ is the cardinality
of $\Gamma_{1}(x)$. The graph $\Gamma$ is called {\it regular} (with
{\it valency} $k$) if each vertex in $X$ has valency $k$.

 A graph $\Gamma$ is said to be {\it distance-regular} whenever for all integers $0 \leq h, i,
j \leq D$, and all vertices $x,y \in X$ with $\partial(x,y)=h$, the
number $$p^h_{ij}=\vert \{\,z\in X \mid\,z\in \Gamma_{i}(x)\cap
\Gamma_{j}(y)\}\vert$$ is independent of $x, y$. The constants
$p^h_{ij}$ are known as the {\it intersection numbers} of $\Gamma$.

Let $\Gamma$=($X$, $R$) be a  distance-regular graph. For two
vertices $x,y\in X$, with $\partial(x,y)=i$, set
\begin{eqnarray*}
  B(x,y)&:=&\Gamma_{1}(x)\cap\Gamma_{i+1}(y),\\
  C(x,y)&:=&\Gamma_{1}(x)\cap\Gamma_{i-1}(y),\\
  A(x,y)&:=&\Gamma_{1}(x)\cap\Gamma_{i}(y).
\end{eqnarray*}
Note that
  \begin{eqnarray*}
   \nonumber \vert B(x,y)\vert &=& p^i_{1\ i+1},\\
   \nonumber \vert C(x,y)\vert &=& p^i_{1\ i-1},\\
   \nonumber \vert A(x,y)\vert &=& p^i_{1\ i}
  \end{eqnarray*}
are independent of $x$, $y$.

For convenience, set $c_i:=p^i_{1\ i-1}$  for $1 \leq i \leq D$,
$a_i:=p^i_{1\ i}$ for $0 \leq i \leq D$, $b_i:=p^i_{1\ i+1}$ for $0
\leq i \leq D-1$ and put $b_D:=0$, $c_0:=0$, $k:=b_0$. Note that $k$
is the valency of $\Gamma$. It is immediate from the definition of
$p^h_{ij}$ that $b_i \neq 0$ for $0\leq i \leq D-1$ and $c_i \neq 0$
for $1\leq i \leq D$. Moreover
\begin{equation}\label{eq2.1}
k=a_i+b_i+c_i \hskip 0.2in {\rm for}~~0\leq i\leq D.
\end{equation}
\medskip

From now on we assume $\Gamma=(X,R)$ is distance-regular with
diameter $D\geq 3$. Recall that a sequence $x$, $y$, $z$ of vertices
of $\Gamma$ is weak-geodetic whenever
\begin{equation*}
  \partial (x,y)+\partial (y,z)\leq \partial (x,z)+1.
\end{equation*}

\begin{defn}\label{N2.1}
Let $\Omega$ be a subset of $X$, and pick any vertex $x\in\Omega$.
$\Omega$ is said to be {\it weak-geodetically closed with respect to
$x$}, whenever for all $z\in\Omega$ and for all $y\in X$,
\begin{equation}\label{eq2.2}
    x, y, z~\text{\rm are~weak-geodetic} ~~ \Longrightarrow ~~  y\in\Omega.
\end{equation}
\end{defn}

Note that $\Omega$ is weak-geodetically closed with respect to a
vertex $x\in \Omega$ if and only if
$$C(z,x)\subseteq\Omega~~ {\rm and}~~ A(z,x)\subseteq\Omega\qquad {\rm for~
all}~~ z\in\Omega$$ \cite[Lemma~2.3]{w:98}. Also $\Omega$ is
weak-geodetically closed if and only if for any vertex $x\in\Omega$,
$\Omega$ is weak-geodetically closed with respect to $x$. We list a
few results which will be used later in this paper.

\medskip

\begin{thm}\label{N2.2}(\cite[Theorem~4.6]{w:98})
Let $\Gamma$ be a distance-regular graph with diameter $D\geq3$. Let
$\Omega$ be a regular subgraph of $\Gamma$ with valency $\gamma$ and
set $d:={\rm min}\{i\mid \gamma\leq c_{i}+a_{i}\}$. Then the
following (i),(ii) are equivalent.
\begin{enumerate}
  \item [(i)]$\Omega$ is weak-geodetically closed with respect to at least one vertex $x\in\Omega$.
  \item [(ii)] $\Omega$ is weak-geodetically closed with diameter $d$.
\end{enumerate}
In this case $\gamma=c_d+a_d$.
\end{thm}

\medskip

\begin{lem}\label{N2.3}(\cite[Lemma~2.6]{s:96})
Let $\Gamma$ be a distance-regular graph with diameter $2$, and let
$x$ be a vertex of $\Gamma$. Suppose $a_2\neq 0$. Then the subgraph
induced on $\Gamma_{2}(x)$ is connected of diameter at most 3.
\end{lem}
\medskip

\begin{thm}\label{N2.4}(\cite[Proposition~6.7]{w:98},\cite[Theorem~1.1]{s:96})
Let $\Gamma$ be a distance-regular graph with diameter $D\geq3$.
Suppose $a_1=0$, $a_2\neq0$ and $\Gamma$ contains no parallelograms
of length $3$. Then $\Gamma$ is $2$-bounded.
\end{thm}
\medskip

\begin{thm}\label{N2.5}(\cite[Lemma~6.9]{w:98},\cite[Lemma~4.1]{s:96})
Let $\Gamma$ be a distance-regular graph with diameter $D\geq3$.
Suppose $a_1=0$, $a_2\neq 0$ and $\Gamma$ contains no parallelograms
of any length. Let $x$ be a vertex of $\Gamma$, and let $\Omega$ be
a weak-geodetically closed subgraph of $\Gamma$ with diameter $2$.
Suppose there exists an integer $i$ and a vertex
$u\in\Omega\cap\Gamma_{i-1}(x)$, and suppose
$\Omega\cap\Gamma_{i+1}(x)\neq\emptyset$. Then for all $t\in\Omega$,
we have $\partial(x,t)=i-1+\partial(u,t)$.
\end{thm}
\bigskip

\section{$Q$-polynomial properties}
 Let $\Gamma=(X,R)$ denote a distance-regular graph with diameter
$D\geq3$. Let $\mathbb{R}$ denote the real number field. Let ${\rm
Mat}_X(\mathbb{R})$ denote the algebra of all the matrices over
$\mathbb{R}$ with the rows and columns indexed by the elements of
$X$. For $0\leq i \leq D$ let $A_i$ denote the matrix in ${\rm
Mat}_X(\mathbb{R})$, defined by the rule
$$
(A_i)_{xy}=\left\{%
\begin{array}{ll}
    1, & \hbox{if $\partial(x, y)=i$;} \\
    0, & \hbox{if $\partial(x, y) \neq i$} \\
\end{array}%
\right.\ \ \ \ \  {\rm for}~~x,y \in X.$$ We call $A_i$ the {\it
distance matrices} of $\Gamma$. We have
\begin{eqnarray*}
 &&A_0  = I,\label{1.2}  \\
 &&A^t_i  =A_i  \ \ \ {\rm for}~0\leq i\leq D~{\rm where}~A^t_i~{\rm means~the~transpose~of}~A_i,\label{1.4}    \\
 &&A_i  A_j ={\sum_{h=0}^D}p^h_{ij}A_h \;\;\;\; {\rm for}\ \ 0\leq i, j\leq D.\label{1.5}    
 \end{eqnarray*}
\medskip

 Let $M$ denote the subspace of ${\rm Mat}_X(\mathbb{R})$
spanned by $A_0,A_1,\ldots,A_D$. Then $M$ is a commutative
subalgebra of ${\rm Mat}_X(\mathbb{R})$, and is known as the {\it
Bose-Mesner algebra } of $\Gamma $. By \cite[p.~59,~64]{bcn}, $M$
has a second basis $E_0, E_1,\ldots,E_D$ such that
\begin{eqnarray}
&&E_0    =\vert X\vert^{-1}J ~~~{\rm where}~J={\rm all} \ 1\text{'s} \ {\rm matrix},\nonumber \\
&&E_iE_j=\delta_{ij}E_i \hskip 0.2in {\rm for}~~0\leq i, j\leq D,\nonumber \\
&&E_0 +E_1+\cdots+E_D=I,\nonumber \\
&&E^t_i = E_i \hskip 0.2in {\rm for}~~0\leq i\leq D.\label{1.10}
\end{eqnarray}
The $E_0,E_1,\ldots,E_D$ are known as the {\it primitive
idempotents} of $\Gamma$, and $E_0$ is known as the {\it trivial}
idempotent. Let $E$ denote any primitive idempotent of $\Gamma$.
Then we have
\begin{equation}\label{1.11}
E=\vert X\vert^{-1}{\sum_{i=0}^D}\theta^*_iA_i
\end{equation}
for some $\theta^*_0,\theta^*_1,\ldots,\theta^*_D \in \mathbb{R}$,
called the {\it dual eigenvalues} associated with $E$.
\bigskip

Set $V=\mathbb{R}^{\vert X \vert}$ (column vectors), and view the
coordinates of $V$ as being indexed by $X$. Then the Bose-Mesner
algebra $M$ acts on $V$ by left multiplication. We call $V$ the {\it
standard module} of $\Gamma$. For each vertex $x \in X$, set
\begin{equation}\label{1.12}
\hat x=(0,0,\ldots,0,1,0,\ldots,0)^t,
\end{equation}
where the $1$ is in coordinate $x$.  Also, let $\langle \, ,\rangle$
denote the dot product
\begin{equation}\label{1.13}\langle u,v \rangle=u^{t}v\ \ \ \ {\rm
for}~~ u,v \in V.\end{equation}

Then referring to the primitive idempotent $E$ in (\ref{1.11}), we
compute from (\ref{1.10})-(\ref{1.13}) that for $x$, $y$ $\in X$,
\begin{equation}\label{1.14} \langle E\hat x,E\hat y \rangle= \vert X
\vert^{-1}\theta^*_i
\end{equation} where $i=\partial (x,y)$.
\bigskip

Let $\circ$ denote the entry-wise multiplication in ${\rm
Mat}_X(\mathbb{R})$. Then
$$A_i\circ A_j=\delta_{ij}A_i  \hskip 0.2in {\rm for}~~0\leq i, j\leq D,$$
\noindent so $M$ is closed under $\circ$. Thus there exists
$q^k_{ij}\in \mathbb{R} \ $ for $0\leq i, j, k\leq D$ such that
$$E_i\circ E_j=\vert X\vert^{-1}{\sum_{k=0}^D}q^k_{ij}E_k \hskip 0.2in {\rm for}~~
 0\leq i, j\leq D.$$
\medskip

$\Gamma$ is said to be {\it $Q$-polynomial} with respect to the
given ordering $E_0$, $E_1$,$\ldots$, $E_D$ of the primitive
idempotents, if for all integers $0\leq h, i, j \leq D$,
$q^h_{ij}=0$ (resp. $ q^h_{ij}\not= 0$) whenever one of $h, i, j$ is
greater than (resp. equal to) the sum of the other two. Let $E$
denote any primitive idempotent of $\Gamma$. Then $\Gamma$ is said
to be $Q$-polynomial with respect to $E$ whenever there exists an
ordering $E_0$, $E_1=E$,$\ldots$, $E_D$ of the primitive idempotents
of $\Gamma$, with respect to which $\Gamma$ is $Q$-polynomial.  If
$\Gamma$ is $Q$-polynomial with respect to $E$, then the associated
dual eigenvalues are distinct \cite[p.~384]{t:92}. \bigskip

The following theorem about the $Q$-polynomial property will be used
in this paper.
\smallskip

\begin{thm} (\cite[Theorem~3.3]{t:95})\label{N3.1}
Assume $\Gamma$ is $Q$-polynomial with respect to a primitive
idempotent $E$, and let $\theta_0^*, \ldots,\theta_D^*$ denote the
corresponding dual eigenvalues. Then for all integers $1\leq h\leq
D$, $0\leq i,j\leq D$ and for all $x,y\in X$ such that
$\partial(x,y)=h$,
\begin{equation} \label{eq3.14} \sum_{\scriptstyle{{z \in X\atop\partial (x,z)=i}}\atop
\partial (y,z)=j}E\hat{z} - \sum_{\scriptstyle{{z \in X\atop\partial (x,z)=j}}\atop
\partial (y,z)=i}E\hat{z}=p^h_{ij}\frac{\theta^*_i - \theta^*_j}{\theta^*_0 - \theta^*_h}(E\hat{x} - E\hat{y}).
\end{equation}
\end{thm}
\bigskip

$\Gamma$ is said to have {\it classical parameters} $(D, b, \alpha,
\beta)$ whenever the intersection numbers of $\Gamma$ satisfy
\begin{eqnarray}
c_i &=& {i\atopwithdelims [] 1}\biggl (1 + \alpha
{i-1\atopwithdelims [] 1}\biggr ) \ \ \ \ \ \ {\rm for}~~0\leq i
\leq D,\label{2.17}\\
b_i &=& \biggl ({D\atopwithdelims [] 1} - {i\atopwithdelims []
1}\biggr ) \biggl (\beta - \alpha {i\atopwithdelims [] 1}\biggr ) \
\ \ \ \ {\rm for}~~0\leq i \leq D,\label{2.18}
\end{eqnarray}
where \begin{equation}\label{2.19}{i\atopwithdelims [] 1} := 1 + b +
b^2 + \cdots +b^{i-1}. \end{equation}
\medskip

The following theorem characterizes the distance-regular graphs
with classical parameters in an algebraic way.

\begin{thm}(\cite[Theorem~4.2]{t:95})\label{N3.2}
 Let $\Gamma$ denote a distance-regular with diameter $D\geq 3$.
 Choose $b\in \mathbb{R}\setminus\{0, -1\},$ and let ${~\atopwithdelims [] ~}$ be as in (\ref{2.19}). Then the following
 (i)-(ii) are equivalent.
 \begin{enumerate}
\item[(i)] $\Gamma$ is $Q$-polynomial with associated dual
eigenvalues $\theta_0^*, \theta_1^*, \ldots,\theta_D^*$ satisfying
\begin{equation}\label{eq3-18}
\theta_i^*-\theta_0^*=(\theta_1^*-\theta_0^*){i\atopwithdelims [] 1}
b^{1-i} \ \ \ \ {\rm for}~~1\leq i \leq D.
\end{equation}

\item[(ii)] $\Gamma$ has classical parameters $(D, b,
\alpha, \beta)$ for some real constants $\alpha, \beta$.
\end{enumerate}
\end{thm}
\medskip

The following theorem characterizes the distance-regular graphs
with classical parameters and $a_1=0$, $a_2\not=0$ in a
combinatorial way.

\begin{thm}(\cite[Theorem~1.1]{plw:06})\label{N3.3}
Let $\Gamma$ denote a distance-regular graph with diameter
$D\geq3$ and intersection numbers $a_1=0$, $a_2\neq 0$. Then the
following (i)-(iii) are equivalent.
\begin{enumerate}
\item[(i)] $\Gamma$ is $Q$-polynomial and contains no
parallelograms of length $3$. \item[(ii)] $\Gamma$ is
$Q$-polynomial and contains no parallelograms of any length $i$
for $3\leq i\leq D$. \item[(iii)] $\Gamma$ has classical
parameters $(D, b, \alpha, \beta)$ for some real constants $b,
\alpha, \beta$.
\end{enumerate}
\end{thm}
\bigskip

\section{Proof of main theorem}
Assume $\Gamma=(X,R)$ is a distance-regular graph with classical
parameters $(D, b, \alpha, \beta)$ and $D\geq3$. Suppose the
intersection numbers $a_1=0$ and $a_2\neq 0$. Then $\Gamma$
contains no parallelograms of any length by Theorem \ref{N3.3}. We
first give a definition.
\begin{defn}\label{N4.1}
For any vertex $x\in X$ and any subset $C\subseteq X$, define
$$[x,C]:=\{v\in X\mid {\rm there\;exists}\;z\in C,\; {\rm such\; that}\; \partial(x,v)+\partial(v,z)=\partial(x,z)\}.$$
\end{defn}

Throughout this section, fix two vertices $x, y\in X$ with
$\partial(x,y)=3$. Set
$$C:=\{z\in\Gamma_3(x)~|~B(x, y)=B(x, z)\}$$
and
\begin{equation}\label{eq4.0}
    \Delta =[x, C].
\end{equation}
 We shall prove $\Delta$ is a regular weak-geodetically closed subgraph of diameter $3$. Note that the
diameter of $\Delta$ is at least $3$. If $D=3$ then $C=\Gamma_3(x)$
and $\Delta=\Gamma$ is clearly a regular weak-geodetically closed
graph. Thereafter we assume $D\geq 4$. By referring to Theorem
\ref{N2.2}, we shall prove $\Delta$ is weak-geodetically closed with
respect to $x$, and the subgraph induced on $\Delta$ is regular with
valency $a_3+c_3.$
\bigskip

\begin{lem}\label{N4.2}
For all adjacent vertices $z, z'\in\Gamma_i(x)$, where
$i\leq D$, we have $B(x, z)=B(x, z')$.
\end{lem}

\begin{proof}
By symmetry, it suffices to show $B(x, z)\subseteq B(x, z')$.
Suppose there exists $w\in B(x, z)\setminus B(x, z')$. Then
$\partial(w,z')\neq i+1$. Note that
$\partial(w,z')\leq\partial(w,x)+\partial(x,z')=1+i$ and
$\partial(w,z')\geq\partial(w,z)-\partial(z,z')=i$. This implies
$\partial(w,z')=i$ and
$wxz'z$ forms a parallelogram of length $i+1$, a contradiction.
\end{proof}

We have known $\Gamma$ is $2$-bound by Theorem \ref{N2.4}. For two
vertices $z,s$ in $\Gamma$ with $\partial(z,s)=2$, let $\Omega(z,s)$
denote the regular weak-geodetically closed subgraph containing
$z,s$ of diameter $2$.
\bigskip

\begin{lem}\label{N4.3} Suppose $stuzw$ is a pentagon
in $\Gamma$, where $s,u\in\Gamma_3(x)$ and $z\in\Gamma_{2}(x)$. Pick
$v\in B(x,u)$. Then $\partial(v,s)\neq 2$.
\end{lem}
\begin{proof}
Suppose $\partial(v,s)= 2$. Note $\partial(z,s)\neq 1$, since
$a_1=0$. Note that $z,w,s,t,u\in\Omega(z,s)$. Then
$s\in\Omega(z,s)\cap\Gamma_{2}(v)$ and
$u\in\Omega(z,s)\cap\Gamma_{4}(v)\neq\emptyset$. Hence
$\partial(v,z)=\partial(v,s)+\partial(s,z)=2+2=4$ by Theorem
\ref{N2.5}. A contradiction occurs since $\partial(v,x)=1$ and
$\partial(x,z)=2$.
\end{proof}

\begin{lem}\label{N4.5}
Suppose $stuzw$ is a pentagon in $\Gamma$, where $s,u\in\Gamma_3(x)$
and $z\in\Gamma_{2}(x)$. Then $B(x,s)=B(x,u)$.
\end{lem}
\begin{proof}
Since $\vert B(x,s)\vert = \vert B(x,u)\vert =b_3$, it suffices to
show $B(x,u)\subseteq B(x,s)$.

\noindent By Lemma \ref{N4.3},
$$B(x,u)\subseteq\Gamma_3(s)\cup\Gamma_{4}(s).$$
Suppose
\begin{eqnarray*}
  \vert B(x,u)\cap\Gamma_3(s)\vert &=& m, \\
  \vert B(x,u)\cap\Gamma_{4}(s)\vert &=& n.
\end{eqnarray*}
Then
\begin{equation}\label{eq4.1}
    m+n=b_3.
\end{equation}
By Theorem \ref{N3.1},
\begin{equation}\label{eq4.2}
\sum_{\scriptstyle{r \in B(x,u)}}E\hat{r} - \sum_{\scriptstyle{r \in
B(u,x)}}E\hat{r}=b_{3}\frac{\theta^*_1 - \theta^*_4}{\theta^*_0 -
\theta^*_3}(E\hat{x} - E\hat{u}).
\end{equation}
 Observe $B(u,x)\subseteq\Gamma_3(s)$, otherwise $\Omega(u,s)\cap B(u,x)\neq\emptyset$ and this leads $\partial(x,s)=4$ by Theorem \ref{N2.5},
 a contradiction. Taking the inner product of $s$ with both side of
(\ref{eq4.2}) and evaluating the result using (\ref{1.14}), we have
\begin{equation}\label{eq4.3}
    m\theta^{*}_{3}+n\theta^{*}_{4}-b_3\theta^{*}_{3}=b_{3}\frac{\theta^*_1 - \theta^*_4}{\theta^*_0 -
\theta^*_3}(\theta^{*}_{3}-\theta^{*}_{2}).
\end{equation}
Solve (\ref{eq4.1}) and (\ref{eq4.3}) to obtain
\begin{equation}\label{eq4.4}
    n=b_{3}\frac{(\theta^*_2-\theta^*_3)}{(\theta^*_3-\theta^*_4)}\frac{(\theta^*_1-\theta^*_4)}{(\theta^*_0-\theta^*_3)}.
\end{equation}
Simplifying (\ref{eq4.4}) by (\ref{eq3-18}), we have $n=b_3$ and
then $m=0$ by (\ref{eq4.1}). This implies $B(x,u)\subseteq B(x,s)$
and ends the proof.
\end{proof}

\begin{lem}\label{N4.6} Let $z,u\in\Delta$. Suppose
$stuzw$ is a pentagon in $\Gamma$, where $z,w\in\Gamma_{2}(x)$ and
$u\in\Gamma_{3}(x)$. Then $w\in\Delta$.
\end{lem}
\begin{proof}
Observe $\Omega(z,s)\cap\Gamma_{1}(x)=\emptyset$ and
$\Omega(z,s)\cap\Gamma_{4}(x)=\emptyset$ by Theorem \ref{N2.5}.
Hence $s,t\in\Gamma_{2}(x)\cup\Gamma_{3}(x)$. Observe
$s\in\Gamma_{3}(x)$, otherwise $w,s\in\Omega(x,z)$, and this implies
$u\in\Omega(x,z)$, a contradiction to that the diameter of
$\Omega(x,z)$ is $2$. Hence $B(x,s)=B(x,u)$ by Lemma \ref{N4.5}.
Then $s\in C$ and $w\in\Delta$ by construction.
\end{proof}
\medskip

\begin{lem}\label{N4.7}
The subgraph $\Delta$ is weak-geodetically closed with respect to
$x$.
\end{lem}
\begin{proof} Clearly $C(z, x)\subseteq \Delta$ for any $z\in
\Delta$. It suffices to show $A(z, x)\subseteq \Delta$ for any
$z\in\Delta$. Suppose $z\in\Delta$. We discuss case by case in the
following. The case $\partial(x,z)=1$ is trivial since $a_1=0$. For
the case $\partial(x,z)=3$, we have $B(x,y)=B(x,z)=B(x,w)$ for any
$w \in A(z,x)$ by definition of $\Delta$ and Lemma \ref{N4.2}. This
implies $A(z, x)\subseteq \Delta$ by the construction of $\Delta$.
For the remaining case $\partial(x,z)=2$, fix $w \in A(z,x)$ and we
shall prove $w\in\Delta$. There exists $u\in C$ such that $z \in
C(u,x)$. Observe that $\partial(w,u)=2$ since $a_1=0$. Choose $s\in
A(w,u)$ and $t\in C(u,s)$. Then $stuzw$ is a pentagon in $\Gamma$.
The result comes immediately by Lemma \ref{N4.6}.
\end{proof}
\medskip

{\noindent \bf Proof of Theorem \ref{main}:}
\smallskip

By Theorem~\ref{N2.2} and Lemma~\ref{N4.7}, it suffices to show
that  $\Delta$ defined in (\ref{eq4.0}) is regular with valency
$a_3+c_3$. Clearly from the construction and Lemma~\ref{N4.7},
$\vert\Gamma_{1}(z)\cap\Delta\vert= a_{3}+c_{3}$ for any $z\in C$.
First we show $\vert\Gamma_{1}(x)\cap\Delta\vert= a_{3}+c_{3}$.
Note that $y\in\Delta\cap\Gamma_3(x)$ by construction of $\Delta$.
For any $z\in C(x,y)\cup A(x,y)$,
$$\partial(x,z)+\partial(z,y)\leq\partial(x,y)+1.$$
This implies $z\in\Delta$ by definition \ref{N2.1}. Hence
$C(x,y)\cup A(x,y)\subseteq\Delta$. Suppose
$B(x,y)\cap\Delta\neq\emptyset$. Choose $t\in B(x,y)\cap\Delta$.
Then there exists $y'\in\Gamma_3(x)\cap\Delta$ such that $t\in
C(x,y')$. Note that $B(x,y)=B(x,y')$. This leads a contradiction to
$t\in C(x,y')$. Hence $B(x,y)\cap\Delta=\emptyset$ and
$\Gamma_1(x)\cap\Delta=C(x,y)\cup A(x,y)$. Then we have
$\vert\Gamma_{1}(x)\cap\Delta\vert= a_{3}+c_{3}$.
\medskip

Since each vertex in $\Delta$ appears in a sequence of vertices
$x=x_0,x_1,x_2,x_3$ in $\Delta$, where $\partial(x,x_j)=j$ and
$\partial(x_{j-1},x_j)=1$ for $1\leq j\leq 3$, it suffices to show
\begin{equation}\label{eq4.5}
    \vert\Gamma_{1}(x_i)\cap\Delta\vert= a_{3}+c_{3}
\end{equation}
for $1\leq i\leq 2$. For each integer $0\leq i\leq 2$, we show
$$\vert\Gamma_{1}(x_{i})\setminus\Delta\vert\leq\vert\Gamma_{1}(x_{i+1})\setminus\Delta\vert$$
by the $2$-way counting of the number of the pairs ($s,z$) for
$s\in\Gamma_{1}(x_{i})\setminus\Delta$,
$z\in\Gamma_{1}(x_{i+1})\setminus\Delta$ and $\partial(s,z)=2$. For
a fixed $z\in\Gamma_{1}(x_{i+1})\setminus\Delta$, we have
$\partial(x,z)=i+2$ by Lemma~\ref{N4.7}, so $\partial(x_i,z)=2$ and
$s\in A(x_i,z)$. Hence the number of such pairs ($s,z$) is at most
$\vert\Gamma_{1}(x_{i+1})\setminus\Delta\vert a_2$.

On the other hand, we show this number is exactly
$\vert\Gamma_{1}(x_{i})\setminus\Delta\vert a_2$. Fix an
$s\in\Gamma_1(x_i)\setminus\Delta$. Observe $\partial(x,s)=i+1$ by
Lemma \ref{N4.7}. Observe $\partial(x_{i+1},s)=2$ since $a_1=0$.
Pick any $z\in A(x_{i+1},s)$. We shall prove $z\not\in\Delta$.
Suppose $z\in\Delta$ in the below arguments and choose any $w\in
C(s,z)$.
\medskip

\noindent Case 1: $i=0$.

Observe $\partial(x,z)=2$, $\partial(x,s)=1$ and $\partial(x,w)=2$.
This will force $s\in\Delta$ by Lemma \ref{N4.7}, a contradiction.
\medskip

\noindent Case 2: $i=1$.

Observe $\partial(x,z)=3$, otherwise $z\in\Omega(x,x_2)$ and this
implies $s\in\Omega(x,x_2)\subseteq\Delta$ by Lemma \ref{N2.3} and
Lemma \ref{N4.7}, a contradiction. This also implies $s\in\Delta$ by
Lemma \ref{N4.7}, a contradiction.
\medskip

\noindent Case 3: $i=2$.

Observe $\partial(x,z)=2$ or $3$. Suppose $\partial(x,z)=2$. Then
$B(x,x_3)=B(x,s)$ by Lemma \ref{N4.5} (with $x_3=u$, $x_2=t$). Hence
$s\in\Delta$, a contradiction. So $z\in\Gamma_{3}(x)$. Note
$\partial(x,w)\neq 2,3$, otherwise $s\in\Delta$ by Lemma \ref{N4.5}
and Lemma \ref{N4.7} respectively. Hence $\partial(x,w)=4$. Then by
applying $\Omega=\Omega(x_2,w)$ in Theorem \ref{N2.5} we have
$\partial(x_2,z)=1$, a contradiction to $a_1=0$.

From the above counting, we have
\begin{equation}\label{eq4.6}
    \vert\Gamma_{1}(x_{i})\setminus\Delta\vert a_2\leq
\vert\Gamma_{1}(x_{i+1})\setminus\Delta\vert a_2
\end{equation}
for $0\leq i\leq 2$. Eliminating $a_2$ from (\ref{eq4.6}), we find
\begin{equation}\label{eq4.7}
    \vert\Gamma_{1}(x_{i})\setminus\Delta\vert\leq
\vert\Gamma_{1}(x_{i+1})\setminus\Delta\vert,
\end{equation}
or equivalently
\begin{equation}\label{eq4.8}
    \vert\Gamma_{1}(x_{i})\cap\Delta\vert\geq
\vert\Gamma_{1}(x_{i+1})\cap\Delta\vert
\end{equation}
for $0\leq i\leq 2$. We have known previously
$\vert\Gamma_1(x_0)\cap\Delta\vert=\vert\Gamma_1(x_3)\cap\Delta\vert
=a_3+c_3$. Hence (\ref{eq4.5}) follows from (\ref{eq4.8}).
 \hfill $\square$
\bigskip

\begin{rem}
The $4$-bounded property seems to be much harder to be proved. We
expect that the $3$-bounded property is enough to classify all the
distance-regular graphs with classical parameters, $a_1=0$ and
$a_2\not=0$.
\end{rem}
\bigskip


\medskip

\noindent Yeh-jong Pan \hfil\break Department of Applied Mathematics
\hfil\break National Chiao Tung University \hfil\break 1001 Ta Hsueh
Road \hfil\break Hsinchu, Taiwan 300, R.O.C.\hfil\break Email: {\tt
yjp.9222803@nctu.edu.tw} \hfil\break Fax: +886-3-5724679 \hfil\break
\medskip

\noindent Chih-wen Weng \hfil\break Department of Applied
Mathematics \hfil\break National Chiao Tung University \hfil\break
1001 Ta Hsueh Road \hfil\break Hsinchu, Taiwan 300,
R.O.C.\hfil\break Email: {\tt weng@math.nctu.edu.tw} \hfil\break
Fax: +886-3-5724679 \hfil\break

\end{document}